\documentclass[12pt]{article}
\usepackage[top=3cm,bottom=3cm,left=2.75cm,right=2.75cm]{geometry}
\usepackage{amssymb}
\usepackage{amsmath,amsthm}
\usepackage[latin1]{inputenc}
\usepackage[dvips]{graphicx}
\usepackage{color}
\usepackage{mathrsfs}
\usepackage{enumerate}
\usepackage{tikz}
\usepackage{tkz-graph}
\usepackage{hyperref}
\usepackage{xifthen}
\usepackage{verbatim}

\hypersetup{colorlinks=true}

\hypersetup{colorlinks=true, linkcolor=blue, citecolor=blue,urlcolor=blue}

\newtheorem{remark}{Remark}[section]

\newtheorem{lemma}[remark]{Lemma}
\newtheorem{theorem}[remark]{Theorem}

\title{On the (independent) semitotal domination in subdivision, middle, and central graphs}
\author{Abel Cabrera-Mart\'inez$^{1}$, Jos\'e Luis L\'opez-Carmona$^{1}$\\[5pt]
Ismael Rios-Villamar$^{2}$, Alejandro Serrano-D\'iaz$^{1}$\\[15pt]
{\small $^{1}$ Universidad de C\'ordoba, Departamento de Matem\'aticas, Campus de Rabanales, 14071,} \\
{\small C\'ordoba, Spain (acmartinez@uco.es, 2locaj@uco.es, aserrano1@uco.es)}\\[7pt]
{\small $^{2}$ Universidad Aut\'onoma de Guerrero,  Facultad de  Matem\'{a}ticas, Carlos E. Adame 54,}\\
{\small  Col. La Garita 39650, Acapulco, Guerrero, Mexico (18305783@uagro.mx)}\\
}

\date{ }
\begin{document}
\maketitle

\begin{abstract}
A dominating set $D$ of a nontrivial connected graph $G$ is called a semitotal dominating set of $G$ if every vertex in $D$ is at distance at most two from another vertex in $D$. If, in addition, $D$ is an independent set, then $D$ is called an independent semitotal dominating set of $G$. 
The (independent) semitotal domination number of $G$ is the minimum cardinality among all (independent) semitotal dominating sets of~$G$. In this paper, we obtain closed formulas for these parameters in the following three well-known graph operators defined from a connected graph: the subdivision, middle, and central graphs. 
\end{abstract}


\vspace{.2cm}


\vspace{.2cm}

\section{Introduction}

Let $G$ be a nontrivial connected graph with vertex set $V(G)$ of order $n=|V(G)|$ and edge set $E(G)$ of size $m=|E(G)|$. Let $V(G)=\{v_1, \ldots , v_n\}$ and $V_E(G)=\{v^{i,j} : v_iv_j\in E(G)\}$ (observe that $v^{i,j}=v^{j,i}$).  In this paper, we study certain graph operators arising structural modifications that incorporate edge subdivision as a common underlying feature. The graph operators considered are described as follows.  	
\begin{itemize}
\item The \emph{subdivision graph} $\mathtt{S}(G)$, explicitly considered in~\cite{Harary1969}, is obtained from $G$ by subdividing each edge exactly~once. Formally, $V(\mathtt{S}(G))=V(G)\cup V_E(G)$ and $E(\mathtt{S}(G))= \{v_iv^{i,j},v_jv^{i,j}: v^{i,j}\in V_E(G)\}$.
\item The \emph{middle graph} $\mathtt{M}(G)$ is obtained from the subdivision graph $\mathtt{S}(G)$ by joining two vertices in $V_E(G)$ whenever their corresponding edges are adjacent in $G$. Formally, $V(\mathtt{M}(G))=V(G)\cup V_E(G)$ and $E(\mathtt{M}(G))=\{v_iv^{i,j},v_jv^{i,j}: v^{i,j}\in V_E(G)\}\cup E(\mathtt{L}(G))$, where $\mathtt{L}(G)$ is the line graph of $G$. Middle graphs were introduced by Hamada and Yoshimura~\cite{Intro-M}. 
\item The \emph{central graph} $\mathtt{C}(G)$ is obtained from the subdivision graph $\mathtt{S}(G)$ by joining two vertices in $V(G)$ whenever they are nonadjacent in $G$. Formally, $V(\mathtt{C}(G))=V(G)\cup V_E(G)$ and $E(\mathtt{C}(G))=E(\overline{G})\cup \{v_iv^{i,j},v_jv^{i,j}: v^{i,j}\in V_E(G)\}$, where $\overline{G}$ is the complement graph of $G$. Central graphs were introduced by Vernold~\cite{Intro-C}.
\end{itemize}
As can be observed, these three graph operators have the same vertex set. Figure~\ref{fig-RMS} illustrates, from left to right, a graph $G$ and the corresponding graphs $\mathtt{S}(G)$, $\mathtt{M}(G)$, and $\mathtt{C}(G)$, respectively.

\vspace{.1cm}

\noindent
In recent years, the study of domination-related parameters in the aforementioned graph operators has received increasing attention. In particular, several studies have investigated domination \cite{Barish,Dom-S-ACM,M-dom}, total domination \cite{Dom-S-ACM,C-total,M-total}, independent domination \cite{ind-midd,Ind-Central,Dom-S-ACM}, and double (total) domination \cite{DD-2026,DTD-2026}, among other parameters. We continue this line of research by considering two further domination parameters, namely semitotal domination and independent semitotal domination.

\begin{figure}[ht]
\centering
\begin{tikzpicture}[scale=.37, transform shape]	
	
\node [draw, shape=circle] (a1) at  (1.5,0) {};
\node [draw, shape=circle] (a2) at  (4.5,0) {};
\node [draw, shape=circle] (a3) at  (4.5,3) {};
\node [draw, shape=circle] (a4) at  (3,5) {};
\node [draw, shape=circle] (a5) at  (1.5,3) {};
			
\draw(a1)--(a5)--(a4)--(a3)--(a2);
\draw (a1)--(a3);
\draw (a2)--(a5);

			
\node [draw, shape=circle] (b1) at  (8,0) {};
\node [draw, shape=circle] (b2) at  (11,0) {};
\node [draw, shape=circle] (b3) at  (11,3) {};
\node [draw, shape=circle] (b4) at  (9.5,5) {};
\node [draw, shape=circle] (b5) at  (8,3) {};

\node [draw, shape=circle] (b51) at  (14,-2.5) {};
\node [draw, shape=circle] (b13) at  (14,-0.5) {};
\node [draw, shape=circle] (b23) at  (14,1.5) {};
\node [draw, shape=circle] (b25) at  (14,3.5) {};
\node [draw, shape=circle] (b34) at  (14,5.5) {};
\node [draw, shape=circle] (b45) at  (14,7.5) {};
			
\draw (b5) to[bend left=-25] (b51);
\draw (b51) to[bend left=15] (b1);
			
\draw (b1) to[bend left=-15] (b13);
\draw (b13) to[bend left=15] (b3);
			
\draw (b2) to[bend left=-10] (b23);
\draw (b23) to[bend left=-10] (b3);
			
\draw (b2) to[bend left=15] (b25);
\draw (b25) to[bend left=-15] (b5);
			
\draw (b3) to[bend left=10] (b34);
\draw (b34) to[bend left=-10] (b4);
			
\draw (b4) to[bend left=15] (b45);
\draw (b45) to[bend left=-40] (b5);

			
\node [draw, shape=circle] (c1) at  (18,0) {};
\node [draw, shape=circle] (c2) at  (21,0) {};
\node [draw, shape=circle] (c3) at  (21,3) {};
\node [draw, shape=circle] (c4) at  (19.5,5) {};
\node [draw, shape=circle] (c5) at  (18,3) {};
			
\node [draw, shape=circle] (c51) at  (24,-2.5) {};
\node [draw, shape=circle] (c13) at  (24,-0.5) {};
\node [draw, shape=circle] (c23) at  (24,1.5) {};
\node [draw, shape=circle] (c25) at  (24,3.5) {};
\node [draw, shape=circle] (c34) at  (24,5.5) {};
\node [draw, shape=circle] (c45) at  (24,7.5) {};

\draw (c5) to[bend left=-25] (c51);
\draw (c51) to[bend left=15] (c1);
		
\draw (c1) to[bend left=-15] (c13);
\draw (c13) to[bend left=15] (c3);
			
\draw (c2) to[bend left=-10] (c23);
\draw (c23) to[bend left=-10] (c3);
			
\draw (c2) to[bend left=15] (c25);
\draw (c25) to[bend left=-15] (c5);
			
\draw (c3) to[bend left=10] (c34);
\draw (c34) to[bend left=-10] (c4);
			
\draw (c4) to[bend left=15] (c45);
\draw (c45) to[bend left=-40] (c5);

\draw (c51) to[bend left=-15] (c13);
\draw (c51) to[bend left=-30] (c25);
\draw (c51) to[bend left=-35] (c45);
			
\draw (c13) to[bend left=-15] (c23);
\draw (c13) to[bend left=-30] (c34);
			
\draw (c25) to[bend left=-30] (c45);
\draw (c25) to[bend left=5] (c23);
			
\draw (c23) to[bend left=-25] (c34);
			
\draw (c34) to[bend left=-15] (c45);

			
\node [draw, shape=circle] (d1) at  (29,0) {};
\node [draw, shape=circle] (d2) at  (32,0) {};
\node [draw, shape=circle] (d3) at  (32,3) {};
\node [draw, shape=circle] (d4) at  (30.5,5) {};
\node [draw, shape=circle] (d5) at  (29,3) {};
			
\draw (d1)--(d2);
\draw (d1)--(d4);
\draw (d4)--(d2);
\draw (d3)--(d5);
			
\node [draw, shape=circle] (d51) at  (35,-2.5) {};
\node [draw, shape=circle] (d13) at  (35,-0.5) {};
\node [draw, shape=circle] (d23) at  (35,1.5) {};
\node [draw, shape=circle] (d25) at  (35,3.5) {};
\node [draw, shape=circle] (d34) at  (35,5.5) {};
\node [draw, shape=circle] (d45) at  (35,7.5) {};

\draw (d5) to[bend left=-25] (d51);
\draw (d51) to[bend left=15] (d1);
			
\draw (d1) to[bend left=-15] (d13);
\draw (d13) to[bend left=15] (d3);
			
\draw (d2) to[bend left=-10] (d23);
\draw (d23) to[bend left=-10] (d3);
			
\draw (d2) to[bend left=15] (d25);
\draw (d25) to[bend left=-15] (d5);
			
\draw (d3) to[bend left=10] (d34);
\draw (d34) to[bend left=-10] (d4);
			
\draw (d4) to[bend left=15] (d45);
\draw (d45) to[bend left=-40] (d5);
			
\end{tikzpicture}
\caption{From left to right, a graph $G$ and its corresponding graphs $\mathtt{S}(G)$, $\mathtt{M}(G)$, and  $\mathtt{C}(G)$, respectively.}\label{fig-RMS}
\end{figure}
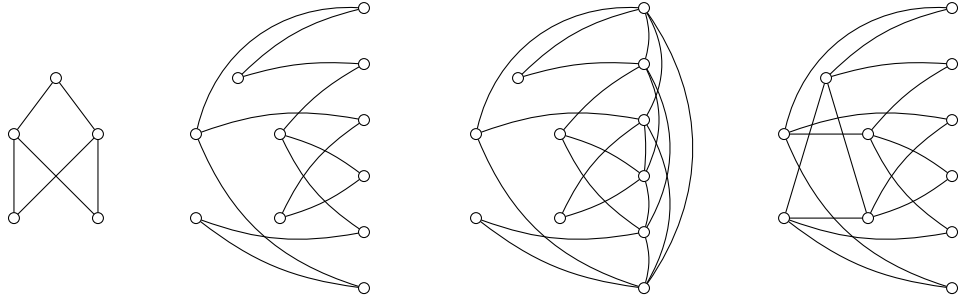

\noindent  
A set $D \subseteq V(G)$ is said to be a \emph{dominating set} of $G$ if $N_G(v_i)\cap D\neq \emptyset$ for every vertex $v_i\in V(G)\setminus D$.
The \emph{domination number} of $G$, denoted by $\gamma(G)$, is  the minimum cardinality among all dominating sets of $G$. A dominating set $D$ satisfying $|D|=\gamma(G)$ is called a $\gamma(G)$-\emph{set}. The same convention will be followed for optimal sets associated with the other parameters and families of sets considered throughout the paper.
For comprehensive surveys on domination in graphs, we refer the reader to~\cite{Haynes1998a,Haynes1998}.
The notion of semitotal domination can be viewed as a natural refinement of classical domination, by imposing an additional condition on the vertices belonging to the dominating set.
A \emph{semitotal dominating set} (SDS) of $G$ is a dominating set $D$ such that every vertex in $D$ is at distance at most two from another vertex in $D$.  
The \emph{semitotal domination number} of $G$, denoted by $\gamma_{t2}(G)$, is  the minimum cardinality among all SDSs of $G$. This parameter was introduced by Goddard et al. in~\cite{GoddardSemiTotal} and has subsequently been studied in several works, including~\cite{std-cite-1,std-cite-2,std-cite-4,std-cite-6}. 

\vspace{.1cm}

\noindent
We next consider the independent counterparts of the domination parameters introduced above. A dominating set $D\subseteq V(G)$ is said to be \emph{independent} if no two vertices in $D$ are adjacent. An \emph{independent dominating set} (IDS) is thus a dominating set that is also independent. The \emph{independent domination number} of $G$, denoted by $i(G)$, is the minimum cardinality among all independent dominating sets of $G$. 
An \emph{independent semitotal dominating set} (ISDS) of $G$ is a semitotal dominating set that is also independent. The \emph{independent semitotal domination number} of $G$, denoted by $i_{t2}(G)$, is the minimum cardinality among all ISDSs of $G$. This parameter was introduced in~\cite{ISD-1} and has subsequently been studied in, among others, \cite{ISD-2027,CLS2026,ISD-2}.

\vspace{.1cm}

\noindent
As previously stated, our goal is to study semitotal domination and independent semitotal domination in the subdivision, middle, and central graphs of a graph $G$. The paper is organized as follows. In Subsection~\ref{sub-1}, we introduce the general notation and terminology, as well as the basic tools needed throughout the paper.
In Section~\ref{section-SM}, we show that $\gamma_{t2}(\mathtt{S}(G))=i_{t2}(\mathtt{S}(G))$ and $\gamma_{t2}(\mathtt{M}(G))=i_{t2}(\mathtt{M}(G))$, and we derive closed formulas for these parameters in terms of invariants of $G$. In Section~\ref{section-C}, we investigate these parameters for the central graph $\mathtt{C}(G)$. In particular, we obtain a closed formula for $\gamma_{t2}(\mathtt{C}(G))$ in terms of the vertex covering number of $G$. Finally, we prove that
$i_{t2}(\mathtt{C}(G)) \in \{i(\mathtt{C}(G)), i(\mathtt{C}(G)) + 1\}$,
and provide examples of families of graphs attaining each of the two possible values.

\subsection{Notation, terminology and useful results}\label{sub-1}

We use the following notation and terminology throughout the paper.
For each $i\in[n]=\{1,\ldots,n\}$, let $N_G(v_i)=\{v_j\in V(G): v_iv_j\in E(G)\}$ and $N_G[v_i]=N_G(v_i)\cup\{v_i\}$ denote the open and closed neighborhoods of $v_i$ in $G$, respectively. A \emph{leaf} of $G$ is a vertex $v_i$ with $|N_G(v_i)|=1$.
If $|N_G(v_i)|=k$ for every $i\in [n]$, then $G$ is called $k$-regular.
For any two vertices $v_i,v_j\in V(G)$, the distance between $v_i$ and $v_j$ in $G$, denoted by $d_G(v_i,v_j)$, is  the length of a shortest $v_i-v_j$ path in $G$. 
For a set~$D\subseteq V(G)$, we denote by $G[D]$ the subgraph of $G$ induced by $D$, and by $G-D$ the subgraph induced by $V(G)\setminus D$.  
We use $K_n$, $C_n$ and $P_n$ to denote the complete graph, the cycle graph and the path graph of order $n$, respectively. Given a graph $G$, the graph $K_1+G$ is the graph with vertex set $V(K_1+G)=V(K_1)\cup V(G)$ and edge set $E(K_1+G)=E(G)\cup \{v_iv_j: v_i\in V(K_1), v_j\in V(G)\}$.
We next introduce the following graph invariants, which will be used throughout the paper.
\begin{itemize}
\item A \emph{vertex cover} of a graph $G$ with at least one edge is a set $D\subseteq V(G)$ such that every edge of $G$ is incident with at least one vertex in $D$. The \emph{vertex covering number} of $G$, denoted by $\beta(G)$, is the minimum cardinality among all vertex covers of $G$.
\item An \emph{edge cover} of graph $G$ with no isolated vertex is a set $F\subseteq E(G)$ such that every vertex of $G$ is incident with at least one edge in $F$. The \emph{edge covering number} of $G$, denoted by $\beta'(G)$, is the minimum cardinality among all edge covers of $G$.
\item A \emph{clique} is a complete subgraph of $G$. The \emph{clique number} of $G$, denoted by $\omega(G)$, is the maximum order among all cliques in $G$. A clique of order three is called a \emph{triangle}, and a graph is said to be \emph{triangle-free} if it contains no triangle. 
\item Given a nontrivial connected graph $G$ of order $n$, let $\mathcal{P}_3(G)$ be the family of sets $D\subseteq V(G)$ such that $G[D]$ is isomorphic to a disjoint union of copies of $P_3$, and every isolated vertex of $G[V(G)\setminus D]$, if any, is a leaf of $G$. Define
$$p_3(G) = \begin{cases} 
0 & \text{if } \mathcal{P}_3(G)=\emptyset, \\
\max\{|D|/3 : D \in \mathcal{P}_3(G)\} & \text{otherwise}.
\end{cases}$$
For instance, $p_3(C_4)=0$ and $p_3(T)>0$ for every tree $T$ of order at least three. If $\mathcal{P}_3(G)\neq \emptyset$, then a set $D\in \mathcal{P}_3(G)$ is called a $p_3(G)$-set if $|D|=3p_3(G)$. 
\end{itemize}

\noindent
We conclude this subsection with two useful theorems. The first one combines two results obtained in \cite{ind-midd} and \cite{M-dom}. The second one, due to Cabrera-Mart\'inez et al.~\cite{Ind-Central},  establishes closed formulas for the independent domination number of $\mathtt{C}(G)$ when $G$ is triangle-free and noncomplete $k$-regular.

\begin{theorem}[\cite{ind-midd} and\cite{M-dom}]\label{teo-domM}
Let $G$ be a nontrivial connected graph. Then
$$\gamma(\mathtt{M}(G))=i(\mathtt{M}(G))= \beta'(G).$$
\end{theorem}

\begin{theorem}{\rm \cite{Ind-Central}}\label{theo-triangle-k-regunlar}
Let $G$ be a connected graph of order $n\geq 3$ and size $m$. The following statements hold.
\begin{enumerate}
\item[{\rm (i)}] If $G$ is a triangle-free graph, then $$i (\mathtt{C}(G))=m+3-\max\{|N_G(v_i)|+|N_G(v_j)|:v_iv_j\in E(G)\}.$$
\item[{\rm (ii)}] If $G$ is a $k$-regular graph, distinct from a complete graph, then $$i (\mathtt{C}(G))=\frac{nk+\omega(G)^2-(2k-1)\omega(G)}{2}.$$
\end{enumerate}
\end{theorem}

\section{Subdivision and middle graphs}\label{section-SM}

We begin with the subdivision graphs. The following theorem shows that the semitotal domination number and the independent semitotal domination number of $\mathtt{S}(G)$ coincide, and expresses their common value in terms of the order of $G$ and the parameter $p_3(G)$.

\begin{theorem}
Let $G$ be  a nontrivial connected graph of order $n$. Then 
$$\gamma_{t2}(\mathtt{S}(G))=i_{t2} (\mathtt{S}(G)) = n- p_3(G).$$
\end{theorem}
	
\begin{proof}
It is easy to check that, if $n \in \{2,3\}$, then $\gamma_{t2}(\mathtt{S}(G))=i_{t2} (\mathtt{S}(G)) = n- p_3(G)$. Henceforth, assume that $n\geq 4$. First, we proceed to prove that $i_{t2} (\mathtt{S}(G)) \leq n - p_3(G)$.
Observe that $V(G)$ is an ISDS of $\mathtt{S}(G)$. If $p_3(G) = 0$, then it follows that $i_{t2} (\mathtt{S}(G)) \leq |V(G)|=n- p_3(G)$, as desired. Thus, we may assume that $p_3(G) >0$. Let $S$ be a $p_3(G)$-set, and let $I = \{v_i \in V(G) \setminus S :\ N_G(v_i) \subseteq S\}$. By definition, the vertices in $I$ are leaves. Since $G$ is a connected graph of order at least four, we have that $V(G) \setminus S \neq \emptyset$. Let $I' \subseteq V_E(G)$ be a set of minimum cardinality such that $N_{\mathtt{S}(G)} (v_i) \cap I' \neq \emptyset$ for every vertex $v_i \in I$. Observe that $|I| = |I'|$ and $|V_E(G[S])| = 2|S|/3$. 
We claim that $X = V_E(G[S]) \cup I' \cup (V(G) \setminus (S\cup I))$ is an ISDS of $\mathtt{S}(G)$. Since the vertices in $I$ are leaves, it is easy to check that $X$ is an IDS of $\mathtt{S}(G)$. Thus, it remains to show that, for every $x\in X$, there exists $y\in X\setminus \{x\}$ such that $d_{\mathtt{S}(G)} (x,y)= 2$. First, let $v_i \in V(G) \setminus (S\cup I)$. 
Since $v_i\notin I$, then there exists a vertex $v_j \in V(G) \setminus (S\cup I)\subseteq X$ such that $v_iv_j \in E(G)$. Hence, $d_{\mathtt{S}(G)}(v_i,v_j) = 2$. 
Next, suppose that $v^{i,j} \in I'$. By the definition of $I'$, we have that $\{v_i,v_j\}\cap S\neq\emptyset$. Without loss of generality, assume that $v_i \in S$. Since every component of $G[S]$ has order at least three, there exists a vertex $v_k \in S$ such that $v_iv_k \in E(G)$. Therefore, $v^{i,k}\in V_E(G[S])\subseteq X$, which implies that $d_{\mathtt{S}(G)} (v^{i,j} , v^{i,k})= 2$. 
Finally, let $v^{i,j} \in V_E(G[S])$. Since $G[S]$ contains disjoint paths $P_3$, there  exists $v_k\in S \setminus \{v_j\}$ such that $v_iv_k\in E(G)$. Thus, $v^{i,k}\in V_E(G[S])\subseteq X$ and consequently, $d_{\mathtt{S}(G)} (v^{i,j}, v^{i,k})= 2$.
Therefore, every vertex in $X$ has another vertex in $X$ at distance two in $\mathtt{S}(G)$, and hence $X$ is an ISDS of $\mathtt{S}(G)$. Consequently, 
$$i_{t2} (\mathtt{S}(G)) \leq  |V_E(G[S]) \cup I' \cup (V(G) \setminus (S\cup I))|=\frac{2|S|}{3} + |I'|+ n -|S| - |I| = n - p_3(G),$$
as desired. Since $\gamma_{t2} (\mathtt{S}(G))\leq i_{t2} (\mathtt{S}(G))$, it remains only to prove that $\gamma_{t2} (\mathtt{S}(G)) \geq n - p_3(G)$. Let $D$ be a $\gamma_{t2} (\mathtt{S}(G))$-set such that $|D\cap V(G)|$ is maximum. We consider the following two complementary cases.
		
\vspace{.2cm}
		
\noindent
Case 1: $V(G)\subseteq D$ or $D\cap V(G) = \emptyset$. 
If $V(G)\subseteq D$, then $\gamma_{t2}(\mathtt{S}(G))=|D|\geq n\geq n-p_3(G)$, and the desired inequality follows.
Suppose that $D\cap V(G)=\emptyset$. This implies that $D=V_E(G)$. If $G$ is not a tree, then $\gamma_{t2} (\mathtt{S}(G)) = |D|=|V_E(G)|\geq n=n-p_3(G)$. If $G$ is a tree, then, since $p_3(G)>0$ and $|V_E(G)|\geq n-1$, it follows that $\gamma_{t2} (\mathtt{S}(G))=|D|=|V_E(G)| \geq n-1 \geq n-p_3(G)$, as desired. 
		
\vspace{.2cm}
		
\noindent
Case 2: $D\cap V(G) \neq \emptyset$ and $V(G) \setminus D\neq \emptyset$. Observe that $V_E(G[V(G) \setminus D]) \subseteq D$. If there exists a component $H$ of $G[V(G) \setminus D]$ such that $|V_E(H)| \geq |V(H)|$, then $D' = (D\setminus V_E(H)) \cup V(H)$ is an SDS of $\mathtt{S}(G)$ with $|D'| \leq |D|$ and $|D' \cap V(G)| > |D\cap V(G)|$, which is a contradiction. Hence, each component of $G[V(G) \setminus D]$ is a tree.
If there exists a component $T$ of order at most two, then $T$ has a vertex $v_i$ which satisfies that $v^{i,j}\in D$ for some $v_j\in V(G)\cap D$. In such a case, we have that $D'' = (D\setminus \{v^{i,j}\}) \cup \{v_i\}$ is an SDS of $\mathtt{S}(G)$ with $|D''|=|D|$ and $|D'' \cap V(G)| > |D\cap V(G)|$, a contradiction too. Therefore, every component $T$ of $G[V(G) \setminus D]$ satisfies that $|V(T)|\geq 3$, which implies that $p_3(T) \geq 1$. If $T_1,\ldots , T_k$ are the components of $G[V(G) \setminus D]$, then $p_3(G)\geq \sum_{i\in [k]}p_3(T_i)\geq k$. Since $V_E(G[V(G) \setminus D]) \subseteq D\cap V_E(G)$, it follows that
\begin{equation*}
|D\cap V_E(G)|\geq |V_E(G[V(G) \setminus D])|
	=\sum_{i\in [k]}|V_E(T_i)|=\sum_{i\in [k]}|V(T_i)|-k\geq |V(G)\setminus D|-p_3(G).
\end{equation*}
Therefore,
$\gamma_{t2} (\mathtt{S}(G)) = |D\cap V(G) | + |D\cap V_E(G)| \geq |D\cap V(G)| + |V(G) \setminus D|-p_3(G)=n-p_3(G),$ which completes the proof.
\end{proof}

\noindent
We next turn our attention to the middle graph. The following theorem shows that the semitotal domination number and the independent semitotal domination number of $\mathtt{M}(G)$ coincide, and expresses their common value in terms of the edge covering number of $G$.

\begin{theorem}
Let $G$ be a connected graph of order at least three. Then
$$\gamma_{t2} (\mathtt{M}(G))=i_{t2} (\mathtt{M}(G)) = \beta'(G).$$
\end{theorem}
	
\begin{proof}
Let $I$ be an $i(\mathtt{M}(G))$-set. We claim that $I$ is also an ISDS of $\mathtt{M}(G)$. To this end, it suffices to show that, for every $x\in I$, there exists a vertex $y\in I\setminus \{x\}$ such that $d_{\mathtt{M}(G)} (x,y)= 2$. First, consider a vertex $v_i \in I$. Let $v_j \in N_G(v_i)$. If $v_j \in I$, then $d_{\mathtt{M}(G)} (v_i,v_j) = 2$. On the other hand, if $v_j\notin I$, then, since $I$ is an IDS of $\mathtt{M}(G)$, there exists a vertex $v_k\in N_G(v_j) \setminus\{v_i\}$ such that $v^{j,k} \in I$. Consequently, $d_{\mathtt{M}(G)} (v_i, v^{j,k}) = 2$. Now, consider a vertex $v^{i,j} \in I$. Since $G$ is a connected graph of order at least three, it follows that $\max\{|N_G(v_i)|, |N_G(v_j)|\} \geq 2$. Without loss of generality, assume that $|N_G(v_i)| \geq 2$, and choose $v_k \in N_G(v_i) \setminus\{v_j\}$. If $v_k \in I$, then $d_{\mathtt{M}(G)} (v^{i,j}, v_k) = 2$. Moreover, if $v_k \notin I$, then there exists a vertex $v_l \in N_G(v_k) \setminus \{v_i\}$ such that $v^{k,l} \in I$. Hence, $d_{\mathtt{M}(G)} (v^{i,j}, v^{k,l}) =2$. Thus, in either case, every vertex in $I$ has another vertex in $I$ at distance two in $\mathtt{M}(G)$. Therefore, $I$ is also an ISDS of $\mathtt{M}(G)$, as claimed. Consequently, $i_{t2} (\mathtt{M}(G)) \leq |I|=i(\mathtt{M}(G))$. This inequality, together with Theorem \ref{teo-domM}, leads to
$$\beta' (G) = \gamma (\mathtt{M}(G)) \leq \gamma_{t2}(\mathtt{M}(G)) \leq i_{t2} (\mathtt{M}(G)) \leq i (\mathtt{M}(G)) = \beta'(G).$$
Therefore, $\gamma_{t2} (\mathtt{M}(G))=i_{t2} (\mathtt{M}(G)) = \beta'(G)$, which completes the proof.
\end{proof}

\section{Central graphs}\label{section-C}

We now turn our attention to the central graphs. To state our first result, we introduce the following family of graphs. A nontrivial connected graph $G$ is said to belong to the family $\mathcal{G}$ if there exists a $\beta(G)$-set that is also a dominating set of $\overline{G}$. The following theorem shows that $\gamma_{t2}(\mathtt{C}(G))$ is either $\beta(G)$ or $\beta(G)+1$, depending on whether or not $G$ belongs to $\mathcal{G}$.

\begin{theorem}
Let $G$ be a nontrivial connected graph. Then 
$$\gamma_{t2} (\mathtt{C}(G)) = 
	\begin{cases}
	 \beta(G) & \text{if } G \in \mathcal{G}, \\
	 \beta(G) + 1 & \text{otherwise.}
    \end{cases}$$
\end{theorem}
	
\begin{proof}
Let $D$ be a $\gamma_{t2}(\mathtt{C}(G))$-set. If $V(G)\subseteq D$ or $G[V(G)\setminus D]$ is isomorphic to an edgeless graph, then $D\cap V(G)$ is a vertex cover of $G$. Consequently, $\beta(G) \leq |D\cap V(G)|\leq |D|=\gamma_{t2}(\mathtt{C}(G))$, as desired. Henceforth, assume that $V(G)\setminus D\neq \emptyset$ and that $G[V(G)\setminus D]$ contains  at least one edge. Observe that $V_E(G[V(G) \setminus D]) \subseteq D \cap V_E(G)$. Let $Q$ be a $\beta(G[V(G) \setminus D])$-set. Hence, $|Q| \leq |V_E(G[V(G) \setminus D])|\leq |D\cap V_E(G)|$. It is straightforward to verify that $(D \cap V(G)) \cup Q$ is a vertex cover of $G$. Therefore, 
$$\beta(G) \leq |D\cap V(G)| + |Q| \leq |D \cap V(G)| + |D\cap V_E(G)| = |D| = \gamma_{t2} (\mathtt{C}(G)).$$
Now, let $S$ be a $\beta(G)$-set, and choose a vertex $v_k \in V(G) \setminus S$. It is easy to check that $S\cup \{v_k\}$ is an SDS of $\mathtt{C}(G)$. Hence, $\gamma_{t2} (\mathtt{C}(G)) \leq |S \cup \{v_k\}| = \beta(G) + 1$. Consequently, $\gamma_{t2} (\mathtt{C}(G))\in \{ \beta(G),\beta(G)+1\}$.
It remains to characterize the case in which $\gamma_{t2} (\mathtt{C}(G))=\beta(G)$. 
Suppose first that $G \in \mathcal{G}$. By definition, there exists a $\beta(G)$-set $S'$ that is also a dominating set of $\overline{G}$, that is, $S'\setminus N_G(v) \neq \emptyset$ for every vertex $v\in V(G) \setminus S'$. This condition implies that $S'$ is also a dominating set of $\mathtt{C}(G)$. Furthermore, it is straightforward to verify that every vertex in $S'$ has another vertex in $S'$ at distance at most two in $\mathtt{C}(G)$. Thus, $S'$ is an SDS of $\mathtt{C}(G)$, and consequently, $\gamma_{t2} (\mathtt{C}(G)) \leq |S'|=\beta(G)$. Since $\gamma_{t2} (\mathtt{C}(G))\in \{ \beta(G),\beta(G)+1\}$, it follows that $\gamma_{t2}(\mathtt{C}(G)) = \beta(G)$, as desired. 
Conversely, suppose that $\gamma_{t2} (\mathtt{C}(G)) = \beta(G)$. Let $D'$ be a $\gamma_{t2} (\mathtt{C}(G))$-set such that $|D'\cap V(G)|$ is maximum. Suppose that $D'\cap V_E(G)\neq \emptyset$. By the maximality of $|D' \cap V(G)|$, it follows that $V_E(G[V(G) \setminus D']) = D' \cap V_E(G)$. Let $Q'$ be a $\beta(G[V(G) \setminus D'])$-set. Then $|Q'| \leq |V_E(G[V(G) \setminus D'])| = |D' \cap V_E(G)|$. Moreover, it is straightforward to verify that $D''=(D' \cap V(G)) \cup Q'$ is an SDS of $\mathtt{C}(G)$. Therefore,
\begin{equation*}
\gamma_{t2} (\mathtt{C}(G)) \leq |D''|=|D'\cap V(G)| + |Q'| \leq |D' \cap V(G)| + |D'\cap V_E(G)| = \gamma_{t2} (\mathtt{C}(G)).
\end{equation*}
Thus, all inequalities in the chain above must in fact be equalities. In particular, $|D''|=\gamma_{t2} (\mathtt{C}(G))$. Hence, $D''$ is a $\gamma_{t2} (\mathtt{C}(G))$-set. Furthermore, since $D'\cap V_E(G)\neq\emptyset$ and $D''\cap V(G)=(D'\cap V(G))\cup Q'$, we have that $|D''\cap V(G)|>|D'\cap V(G)|$, contradicting the maximality of $|D'\cap V(G)|$. Consequently, $D'\cap V_E(G)=\emptyset$.
This implies that $D'\subseteq V(G)$ and that $D'$ is a vertex cover of $G$. Moreover, since $|D'|=\gamma_{t2} (\mathtt{C}(G))=\beta(G)$, we conclude that $D'$ is a $\beta(G)$-set. As $D'$ is also a dominating set of $\overline{G}$, it follows that $G \in \mathcal{G}$, which completes the proof.
\end{proof}
	
\noindent
We now turn to the independent semitotal domination number of the central graph. The following theorem shows that $i_{t2}(\mathtt{C}(G))$ is either $i(\mathtt{C}(G))$ or $i(\mathtt{C}(G))+1$.

\begin{theorem}\label{theo-it2-i}
Let $G$ be a nontrivial connected graph. Then
$$i_{t2}(\mathtt{C}(G)) \in \{i(\mathtt{C}(G)), i(\mathtt{C}(G)) + 1\}.$$
\end{theorem}

\begin{proof}
If $G \cong P_2$, then it is easy to check that $i_{t2}(\mathtt{C}(P_2)) = 2 = i(\mathtt{C}(P_2)) + 1$. Thus, assume that $|V(G)| \geq 3$. By definition, $i_{t2}(\mathtt{C}(G)) \geq i(\mathtt{C}(G))$. Therefore, it remains to prove that $i_{t2}(\mathtt{C}(G)) \leq i(\mathtt{C}(G)) + 1$. Let $I$ be an $i(\mathtt{C}(G))$-set, and define $I^*=\{x\in I: d_{\mathtt{C}(G)}(x,y)\geq 3 \text{ for every vertex } y\in I\setminus \{x\}\}$. 
If $I$ is also an ISDS of $\mathtt{C}(G)$, then $i_{t2}(\mathtt{C}(G)) \leq |I| = i(\mathtt{C}(G))$, as desired. Henceforth, assume that $I$ is not an ISDS of $\mathtt{C}(G)$. Then $I^*\neq \emptyset$. We now consider the following two complementary cases.

\vspace{0.2cm}

\noindent 
Case 1: $I^*\cap V_E(G)\neq \emptyset$. Let $v^{k,l} \in I^*$. Observe that $J = (I \setminus \{v^{k,l}\}) \cup \{v_k, v_l\}$ is an IDS of $\mathtt{C}(G)$. Moreover, it is easy to check that every vertex $v^{i,j} \in J\cap V_E(G)$ satisfies that $d_{\mathtt{C}(G)}(v^{i,j}, v_k)=2$. 
Since $G[J \cap V(G)]$ is a clique of order at least two, we also have that $d_{\mathtt{C}(G)}(v_i, v_j) = 2$ for every two distinct vertices $v_i, v_j \in J \cap V(G)$. Therefore, $J$ is an ISDS of $\mathtt{C}(G)$, and consequently, $i_{t2}(\mathtt{C}(G)) \le |J| = |I| + 1 = i(\mathtt{C}(G)) + 1$, as desired.

\vspace{0.2cm}

\noindent 
Case 2: $I^*\subseteq V(G)$. Let $v_k \in I^*$. Since $G[I \cap V(G)]$ is a clique, we obtain that $I \cap V(G)=I^*=\{v_k\}$. Suppose that there exists a vertex $v_i\in V(G)\setminus N_G[v_k]$. Then, for every $v_j\in N_G(v_i)$, we have that $v^{i,j}\in I$. Moreover, $d_{\mathtt{C}(G)}(v_k, v^{i,j}) = 2$, contradicting the fact that $v_k\in I^*$. Therefore, $N_G[v_k]=V(G)$. 
Let $C\subseteq V(G)\setminus \{v_k\}$ be a set of maximum cardinality such that $G[C]$ is a clique, and define $C' = \{v^{i,j} \in I \cap V_E(G) : \{v_i, v_j\} \cap C \neq \emptyset\}$. Observe that $|C| \le |C'| + 1$. By the maximality of $|C|$, it is straightforward to verify that $J=(I \setminus C') \cup C$ is an ISDS of $\mathtt{C}(G)$. Hence, $i_{t2}(\mathtt{C}(G)) \leq |J| = |I| - |C'| + |C| \leq |I| + 1 = i(\mathtt{C}(G)) + 1$, as desired. 

\vspace{0.2cm}

\noindent 
From the preceding two cases, we conclude that $i_{t2}(\mathtt{C}(G)) \leq i(\mathtt{C}(G)) + 1$, which completes the proof.
\end{proof}

\noindent
We next consider the independent semitotal domination number of the central graph $\mathtt{C}(G)$ when $G$ is triangle-free or noncomplete $k$-regular. The following two theorems provide closed formulas for $i_{t2}(\mathtt{C}(G))$ under these assumptions.

\begin{theorem}
Let $G$ be a connected triangle-free graph of size $m\geq 2$. Then 
$$i_{t2} (\mathtt{C}(G)) =m+3-\max\{|N_G(v_i)|+|N_G(v_j)|:v_iv_j\in E(G)\}.$$
\end{theorem}
	
\begin{proof}
By Theorem~\ref{theo-triangle-k-regunlar}-(i), we only need to prove that $i_{t2}(\mathtt{C}(G))=i(\mathtt{C}(G))$. By definition, $i_{t2}(\mathtt{C}(G)) \geq i(\mathtt{C}(G))$. Therefore, it remains to prove that $i_{t2}(\mathtt{C}(G)) \leq i(\mathtt{C}(G))$. Let $I$ be an $i(\mathtt{C}(G))$-set. Since $G$ is a triangle-free graph and $G[I \cap V(G)]$ is a clique, we have that $|I\cap V(G)|\leq 2$. We consider the following three complementary cases.

\vspace{.2cm}

\noindent
Case 1: $I\cap V(G)=\emptyset$. In this case, $I= V_E(G)$. Since $G$ is a connected graph of order at least three, it follows that $I$ is also an ISDS of $\mathtt{C}(G)$. Therefore, $i_{t2} (\mathtt{C}(G)) \leq |I| = i(\mathtt{C}(G))$, as desired.

\vspace{.2cm}
		
\noindent
Case 2: $I \cap V(G)=\{v_k\}$. If $N_G[v_k]=V(G)$, then, since $G$ is triangle-free, $G$ is isomorphic to a star graph. Consequently, $N_{\mathtt{C}(G)}(v_i) \cap I = \emptyset$ for every $v_i\in V(G) \setminus \{v_k\}$, a contradiction. Hence, $N_G[v_k]\neq V(G)$, which implies that $I\cap V_E(G)\neq \emptyset$.
Let $v^{i,j} \in I\cap V_E(G)$. Since $G$ is triangle-free, it follows that $v_iv_k\notin E(G)$ or $v_jv_k \notin E(G)$. In either case, we obtain that $d_{\mathtt{C}(G)} (v^{i,j}, v_k) = 2$. Thus, $I$ is also an ISDS of $\mathtt{C}(G)$. Therefore, $i_{t2} (\mathtt{C}(G)) \leq |I| = i(\mathtt{C}(G))$, as desired.

\vspace{.2cm}
		
\noindent
Case 3: $I \cap V(G)=\{v_k,v_l\}$ ($k\neq l$). In this case, $d_{\mathtt{C}(G)} (v_k,v_l) = 2$. If $I=\{v_k,v_l\}$, then $I$ is also an ISDS of $\mathtt{C}(G)$, and thus, $i_{t2} (\mathtt{C}(G)) \leq |I| = i(\mathtt{C}(G))$, as desired. Assume that $I\cap V_E(G)\neq \emptyset$. Proceeding analogously to Case 2, we obtain that every vertex $v^{i,j} \in I$ satisfies that $d_{\mathtt{C}(G)} (v^{i,j}, v_k) = 2$. Thus, $I$ is also an ISDS of $\mathtt{C}(G)$. Therefore, $i_{t2} (\mathtt{C}(G)) \leq |I| = i(\mathtt{C}(G))$, as desired.

\vspace{.2cm}
		
\noindent 
From the preceding three cases, we conclude that $i_{t2}(\mathtt{C}(G))=i(\mathtt{C}(G))$, which completes the proof.
\end{proof}

\begin{theorem}
Let $G$ be a connected $k$-regular graph, distinct from a complete graph. Then
$$i_{t2}(\mathtt{C}(G))=\frac{nk+\omega(G)^2-(2k-1)\omega(G)}{2}.$$
\end{theorem}
	
\begin{proof}
By Theorem~\ref{theo-triangle-k-regunlar}-(ii), we only need to prove that $i_{t2}(\mathtt{C}(G))=i(\mathtt{C}(G))$. By definition, $i_{t2}(\mathtt{C}(G)) \geq i(\mathtt{C}(G))$. Therefore, it remains to prove that $i_{t2}(\mathtt{C}(G)) \leq i(\mathtt{C}(G))$. Let $I$ be an $i(\mathtt{C}(G))$-set. We consider the following three complementary cases.

\vspace{.2cm}
		
\noindent
Case 1: $I\cap V(G)=\emptyset$. In this case, $I= V_E(G)$. Since $G$ is a connected graph of order at least four, it follows that $I$ is also an ISDS of $\mathtt{C}(G)$. Therefore, $i_{t2} (\mathtt{C}(G)) \leq |I| = i(\mathtt{C}(G))$, as desired.

\vspace{.2cm}
		
\noindent
Case 2: $I \cap V(G)=\{v_r\}$. If $N_G[v_r]=V(G)$, then, since $G$ is $k$-regular, $G$ is isomorphic to a complete graph, contradicting the assumption. Hence, $N_G[v_r]\neq V(G)$, which implies that $I\cap V_E(G)\neq \emptyset$. 
First, we prove that there exists a vertex $y\in I$ such that $d_{\mathtt{C}(G)}(v_r, y)= 2$. Let $v_s\in V(G)\setminus N_G[v_r]$. Since $G$ is a connected graph of order at least four, there exists a vertex $v_t\in N_G(v_s)$. Hence, $v^{s,t}\in I$ and satisfies that $d_{\mathtt{C}(G)} (v_r, v^{s,t})=2$, as desired. Now, we prove that every vertex in $I\cap V_E(G)$ has another vertex in $I$ at distance two in $\mathtt{C}(G)$. Let $v^{i,j} \in I\cap V_E(G)$. If $|N_G(v_r)\cap \{v_i,v_j\}|\leq 1$, then $d_{\mathtt{C}(G)} (v_r,v^{i,j}) = 2$. Henceforth, assume that $v_i,v_j\in N_G(v_r)$. If $N_G(v_i) = \{v_j,v_r\}$ and $N_G(v_j) = \{v_i,v_r\}$, then, since $G$ is a connected $k$-regular graph, $G \cong K_3$, a contradiction. Hence, $N_G(v_i) \neq \{v_j,v_r\}$ or $N_G(v_j) \neq  \{v_i,v_r\}$. Without loss of generality, assume that $N_G(v_i) \neq \{v_j,v_r\}$. Then there exists a vertex $v_l \in N_G(v_i) \setminus \{v_j,v_r\}$. Observe that $v^{i,l}\in I$ and $d_{\mathtt{C}(G)} (v^{i,j}, v^{i,l}) =2$. Thus, $I$ is also an ISDS of $\mathtt{C}(G)$. Therefore, $i_{t2} (\mathtt{C}(G)) \leq |I| = i(\mathtt{C}(G))$, as desired.

\vspace{.2cm}
		
\noindent
Case 3:  $|I\cap V(G)|\geq 2$. In this case, we have that $d_{\mathtt{C}(G)} (v_i,v_j) = 2$ for any two distinct vertices $v_i,v_j\in I\cap V(G)$. 
If $I\subseteq V(G)$, then $I$ is also an ISDS of $\mathtt{C}(G)$, and thus, $i_{t2} (\mathtt{C}(G)) \leq |I| = i(\mathtt{C}(G))$, as desired. Assume that $I\cap V_E(G)\neq \emptyset$. 
Now, we prove that every vertex in $I\cap V_E(G)$ has another vertex in $I$ at distance two in $\mathtt{C}(G)$. 
Let $v^{i,j} \in I\cap V_E(G)$. If $N_G(v_i) = (I\cap V(G))\cup \{v_j\}$ and $N_G(v_j) = (I\cap V(G)) \cup \{v_i\}$, then, since $G$ is a connected $k$-regular graph, $k = |I\cap V(G)| + 1$. Observe that $((I\cap V(G)) \setminus\{v_l\} ) \cup \{v_i,v_j\} \subseteq N_G(v_l)$ for any $v_l \in I\cap V(G)$. Hence,  $N_G(v_l) = ((I\cap V(G)) \setminus\{v_l\} ) \cup \{v_i,v_j\}$. Consequently, $G$ is isomorphic to a complete graph, which contradicts the assumption. Thus, $N_G(v_i) \neq (I\cap V(G))\cup \{v_j\}$ or $N_G(v_j) \neq (I\cap V(G)) \cup \{v_i\}$. Without loss of generality, assume that $N_G(v_i) \neq (I\cap V(G))\cup \{v_j\}$.
If there exists a vertex $v_s \in   (I\cap V(G))\setminus N_G (v_i)$, then $d_{\mathtt{C}(G)} (v^{i,j}, v_s) =2$, as desired. On the other hand, if $I\cap V(G) \subseteq N_G(v_i)$, then there exists a vertex $v_t \in N_G(v_i) \setminus ((I\cap V(G)) \cup \{v_j\})$ such that $v^{i,t}\in I$ and $d_{\mathtt{C}(G)} (v^{i,j}, v^{i,t}) = 2$, as desired. Therefore, $I$ is also an ISDS of $\mathtt{C}(G)$. Thus, $i_{t2} (\mathtt{C}(G)) \leq |I| = i(\mathtt{C}(G))$, as desired.

\vspace{.2cm}
		
\noindent 
From the preceding three cases, we conclude that $i_{t2}(\mathtt{C}(G))=i(\mathtt{C}(G))$, which completes the proof.
\end{proof}

\noindent
The following lemma will be useful in proving the relation between $i_{t2}(\mathtt{C}(K_1+G))$ and $i_{t2}(\mathtt{C}(G))$ given in the subsequent theorem.

\begin{lemma}\label{lem-join}
Let $G$ be a nontrivial connected graph. Then there exists an $i_{t2}(\mathtt{C}(G))$-set $I$ such that $I\cap V(G) \neq \emptyset$.
\end{lemma}
	
\begin{proof}
If $G\cong P_2$, then the result is immediate. Henceforth, assume that $|V(G)|\geq 3$. Let $I$ be an $i_{t2}(\mathtt{C}(G))$-set such that $|I\cap V(G)|$ is maximum. Suppose that $I\cap V(G) = \emptyset$. This implies that $I= V_E(G)$. Let $C\subseteq V(G)$ be a set of maximum cardinality such that $G[C]$ is a clique. Let $Y = \{v^{i,j}\in V_E(G) :\ \{v_i,v_j\} \cap C \neq \emptyset\}$. Observe that $|Y| \geq |C|$. It is straightforward to verify that $I' = (I\setminus Y) \cup C$ is an ISDS of $\mathtt{C}(G)$. Hence, $|I'| = |I| - |Y| + |C| \leq |I| = i_{t2}(\mathtt{C}(G))$ and $|I' \cap V(G)|> |I\cap V(G)|$, which is a contradiction. Therefore, $I\cap V(G) \neq \emptyset$, which completes the proof.
\end{proof}

\begin{theorem}\label{theo-K1+G}
Let $G$ be a nontrivial connected graph. Then 
$$i_{t2} (\mathtt{C}(K_1 + G)) = i_{t2} (\mathtt{C}(G)) + 1.$$
\end{theorem}
	
\begin{proof}
Let $V(K_1 +G) = \{v_1,\dots, v_n, v_{n+1}\}$, where $V(G) = \{v_1,\dots, v_n\}$ and $V(K_1)= \{v_{n+1}\}$. Observe that $i_{t2}(\mathtt{C}(K_1 + P_2))=3 = i_{t2} (\mathtt{C}(P_2)) + 1$, as desired. Henceforth, assume that $n\geq 3$. We first prove that $i_{t2} (\mathtt{C}(K_1 + G)) \leq i_{t2} (\mathtt{C}(G)) + 1$.  Let $J$ be an $i_{t2}(\mathtt{C}(G))$-set satisfying the condition given in Lemma~\ref{lem-join}. Since $J\cap V(G)\neq \emptyset$, it is straightforward to verify that $J \cup \{v_{n+1}\}$ is an ISDS of $\mathtt{C}(K_1 + G)$. Therefore, $i_{t2} (\mathtt{C}(K_1+G)) \leq |J|+1 = i_{t2} (\mathtt{C}(G)) + 1$, as desired. It remains to prove that $i_{t2} (\mathtt{C}(K_1 + G)) \geq i_{t2} (\mathtt{C}(G)) + 1$. Among all $i_{t2}(\mathtt{C}(K_1+G))$-sets satisfying the condition given in Lemma~\ref{lem-join}, let $I$ be one for which $|I\cap V(K_1+G)|$ is maximum. We consider the following two complementary cases.

\vspace{.2cm}
		
\noindent
Case 1: $v_{n+1} \in I$. Let $v^{i,j} \in I\cap V_E(G)$. Since $d_{\mathtt{C}(K_1+G)} (v^{i,j}, v_{n+1}) = 3$, there exists a vertex $x \in I \cap V(\mathtt{C}(G))$ such that $d_{\mathtt{C}(G)} (v^{i,j},x) = 2$. Observe that $I\cap V(G) \neq \emptyset$. 
We claim that every vertex in $I\cap V(G)$ has another vertex in $I \cap V(\mathtt{C}(G))$ at distance two. Let $v_k\in I\cap V(G)$. If there exists a vertex $v_i\in (I\cap V(G))\setminus \{v_k\}$, then, since $G[I \cap V(G)]$ is a clique, we have that $d_{\mathtt{C}(G)} (v_i,v_k) = 2$, as desired. On the other hand, assume that $I\cap V(G)=\{v_k\}$. Suppose that $N_G[v_k]=V(G)$. Since $I$ is an ISDS of $\mathtt{C}(K_1 + G)$ and $|N_{K_1+G} (v_{n+1}) | = |N_{K_1 + G} (v_k) | = n$, each component of $G - \{v_k\}$ has order at least three. Let $C\subseteq V(G)\setminus \{v_k\}$ be a set of maximum cardinality such that $G[C]$ is a clique, and let $Y = \{v^{i,j} \in V_E(G-\{v_k\}) :\ \{v_i,v_j\} \cap C\neq \emptyset\}$. Observe that $|C| \leq |Y|$. It is straightforward to verify that $I' = (I\setminus Y) \cup C$ is an ISDS of $\mathtt{C}(K_1 + G)$. Hence, $|I'| = |I| - |Y| + |C| \leq |I| = i_{t2} (\mathtt{C}(K_1 + G))$. Consequently, $I'$ is an $i_{t2}(\mathtt{C}(K_1+G))$-set satisfying the condition given in Lemma~\ref{lem-join} such that $|I'\cap V(G) | > |I\cap V(G)|$, which contradicts the choice of $I$. Thus, $N_G[v_k]\neq V(G)$. Therefore, there exists a vertex $v_l\in V(G)\setminus N_G[v_k]$. Since $G$ is a connected graph of order at least three, there exists a vertex $v_r\in N_G(v_l)$. Observe that $v^{l,r}\in I$ and $d_{\mathtt{C}(G)} (v^{l,r}, v_k) =2$, as desired. In either case, every vertex in $I\cap V(G)$ has another vertex in $I \cap V(\mathtt{C}(G))$ at distance two.  Consequently, $I\setminus\{v_{n+1}\}$ is an ISDS of $\mathtt{C}(G)$. Therefore, $i_{t2} (\mathtt{C}(K_1 + G)) =  |I\setminus \{v_{n+1}\}| +1 \geq i_{t2} (\mathtt{C}(G)) + 1$, as desired.

\vspace{.25cm}
		
\noindent
Case 2: $v_{n+1} \notin I$. By Lemma \ref{lem-join}, we have that $I\cap V(K_1+G)\neq \emptyset$. Consequently, $I\cap V(G) \neq \emptyset$. Since $N_{\mathtt{C}(K_1 +G)} (v_{n+1}) \cap I \neq \emptyset$, we deduce that $V(G) \setminus I \neq \emptyset$. Now, we consider the following two subcases.

\vspace{.2cm}

\noindent
Subcase 2.1: $V(G) \setminus I=\{v_k\}$. Observe that $N_{\mathtt{C}(K_1+G)}(v_{n+1})\cap I=\{v^{k,n+1}\}$. 
Since $|V(G) \cap I|=n-1\geq 2$, it is straightforward to verify that $I\setminus\{v^{k,n+1}\}$ is an ISDS of $\mathtt{C}(G)$. Therefore, $i_{t2} (\mathtt{C}(K_1+G)) = |I\setminus \{v^{k,n+1}\}| +1 \geq i_{t2} (\mathtt{C}(G)) + 1$, as desired. 

\vspace{.2cm}

\noindent
Subcase 2.2: $|V(G) \setminus I| \geq 2$. Without loss of generality, assume that $V(G) \setminus I = \{v_1,\ldots, v_k\}$. Observe that $\{v^{1,n+1}, \dots, v^{k,n+1}\} \subseteq I$. Let $B = \{v_i \in V(G) \setminus I :\ N_G(v_i) = I\cap V(G) \}$. If $B = \emptyset$, then it is straightforward to verify that $I\setminus\{v^{1,n+1}, \dots, v^{k,n+1}\}$ is an IDS of $\mathtt{C}(G)$. Therefore, using this fact together with Theorem~\ref{theo-it2-i}, we obtain that $i_{t2} (\mathtt{C}(K_1+G)) = |I\setminus\{v^{1,n+1}, \dots, v^{k,n+1}\}| +k \geq i (\mathtt{C}(G)) + 2\geq i_{t2} (\mathtt{C}(G)) + 1$, as desired. On the other hand, assume that $B\neq \emptyset$. Without loss of generality, suppose that $v_1\in B$. Since $V(G) \cap I \neq \emptyset$, it is easy to check that $(I\setminus \{v^{1,n+1}, \dots, v^{k,n+1}\}) \cup \{v_1\}$ is an ISDS of $\mathtt{C}(G)$. Therefore, $i_{t2} (\mathtt{C}(K_1+G)) = |(I\setminus\{v^{1,n+1}, \dots, v^{k,n+1}\})\cup \{v_1\}| +k-1 \geq i_{t2} (\mathtt{C}(G)) + 1$, as desired.

\vspace{.25cm}
		
\noindent
From the two previous cases, we conclude that $i_{t2} (\mathtt{C}(K_1+G))=i_{t2} (\mathtt{C}(G)) +1$, which completes the proof.
\end{proof}

\noindent
Finally, the following theorem provides infinite many graphs attaining each of the two possible values of $i_{t2}(\mathtt{C}(G))$. More precisely, it shows that whether $i_{t2}(\mathtt{C}(G))$ equals $i(\mathtt{C}(G))$ or $i(\mathtt{C}(G))+1$ is preserved when passing from $G$ to $K_1+G$.

\begin{theorem}
Let $G$ be a connected graph of order at least three. Then the following statements hold.
\begin{enumerate}
\item[{\rm (i)}] If $i_{t2}(\mathtt{C}(G))=i(\mathtt{C}(G))$, then $i_{t2}(\mathtt{C}(K_1+G))=i(\mathtt{C}(K_1+G))$.
\item[{\rm (ii)}] If $i_{t2}(\mathtt{C}(G))=i(\mathtt{C}(G))+1$, then $i_{t2}(\mathtt{C}(K_1+G))=i(\mathtt{C}(K_1+G))+1$.
\end{enumerate}
\end{theorem} 

\begin{proof}
It was established in \cite{Ind-Central} that $i(\mathtt{C}(K_1+G))=i(\mathtt{C}(G))+1$. Furthermore, by Theorem~\ref{theo-K1+G}, we have that $i_{t2}(\mathtt{C}(K_1+G))=i_{t2}(\mathtt{C}(G))+1$. Consequently, for each $k\in \{0,1\}$, if $i_{t2}(\mathtt{C}(G))=i(\mathtt{C}(G))+k$, then 
$$i_{t2}(\mathtt{C}(K_1+G))=i_{t2}(\mathtt{C}(G))+1=i(\mathtt{C}(G))+k+1=i(\mathtt{C}(K_1+G))+k.$$
Setting $k=0$ and $k=1$ gives statements (i) and (ii), respectively, which completes the proof.
\end{proof}

\end{document}